\documentclass[12pt,amstex]{amsart}
\usepackage{etex}
\usepackage{txfonts}

\usepackage{epsfig}
\usepackage{amsmath}
\usepackage{amssymb}
\usepackage{amscd}
\usepackage{graphicx}
\usepackage[all]{xy}
\usepackage{pstricks}
\usepackage{pst-node}
\usepackage[english]{babel}
\usepackage{float}
\usepackage{tikz}
\usetikzlibrary{matrix}
\newtheorem*{thm*}{Theorem}
\newtheorem{thm}{Theorem}[section]
\newtheorem{lem}[thm]{Lemma}

\newtheorem{prop}[thm]{Proposition}
\newtheorem{ques}[thm]{Question}
\newtheorem{cor}[thm]{Corollary}

\newtheorem{conv}[thm]{Convention}

\theoremstyle{definition}
\newtheorem{defn}[thm]{Definition}

\theoremstyle{remark}
\newtheorem{rem}[thm]{Remark}

\numberwithin{equation}{section}

\def \Q {{\bf Q}}
\def \wh {\widehat}
\def \to {\rightarrow}

\def \N {\mathbb N}

\def \Z {\mathbb Z}

\def \d {\delta}
\def \l {\lambda}
\def \x {\bold{x}}
\def \I {\mathcal{I}}
\def \id {{\rm id}}

\newcommand*\normm[1]{\left|\mspace{-1mu}\left|\mspace{-1mu}\left|#1\right|\mspace{-1mu}\right|\mspace{-1mu}\right| }

\def \W {\mathcal{W}}

\begin{document}

\title{A pointwise cubic average for two commuting transformations}
\author{Sebasti\'an Donoso and Wenbo Sun}

 \address{Centro de Modelamiento Matem\'atico and Departamento de Ingenier\'{\i}a
Matem\'atica, Universidad de Chile, Av. Blanco Encalada 2120,
Santiago, Chile \newline  Universit\'e Paris-Est, Laboratoire d'analyse et de math\'ematiques
appliqu\'ees, 5 bd Descartes, 77454 Marne la Vall\'ee
Cedex 2, France} \email{sdonoso@dim.uchile.cl, sebastian.donoso@univ-paris-est.fr }

\address{Department of Mathematics, Northwestern University, 2033
Sheridan Road Evanston, IL 60208-2730, USA}
 \email{swenbo@math.northwestern.edu}
\maketitle

\begin{abstract}
Huang, Shao and Ye 
recently studied pointwise multiple averages by using suitable topological models. Using a notion of dynamical cubes introduced by the authors, the Huang-Shao-Ye technique and the Host machinery of magic systems, we prove that for a system $(X,\mu,S,T)$ with commuting transformations $S$ and $T$, the average \[\frac{1}{N^2} \sum_{i,j=0}^{N-1} f_0(S^i x)f_1(T^j x)f_2(S^i T^j x)\]
converges a.e. as $N$ goes to infinity for any $f_0,f_1,f_2\in L^{\infty}(\mu)$.
\end{abstract}

\section{Introduction}

\subsection{Pointwise convergence for cube averages}
A system $(X,\mathcal{X},\mu,S,T)$ with two commuting transformations $S$ and $T$ is a probability space $(X,\mathcal{X},\mu)$ endowed with two commuting measure preserving transformations $S,T \colon$ $X \to X$.
In this paper, we study the pointwise convergence of a cubic average in such a system.

The existence of the limit in $L^2$ of the averages
\begin{equation} \label{EQ1} \lim_{N\to\infty}\frac{1}{N^2} \sum_{i,j=0}^{N-1} f_0(T^i x)f_1(T^j x)f_2(T^{i+j}x)
 \end{equation}
 was proved by Bergelson \cite{B} and was generalized in \cite{HK04} and \cite{HK} to higher orders averages.
There are two possible generalizations of these averages to systems with commuting transformations: one is to study averages of the form
\begin{equation} \label{EQ1} \lim_{N\to\infty}\frac{1}{N^2} \sum_{i,j=0}^{N-1} f_0(S^i x)f_1(T^j x)f_2(R^{i+j}x)
 \end{equation}
for commuting transformations $S,T$ and $R$. Another is to study averages of the form
\begin{equation} \label{EQ2} \lim_{N\to\infty}\frac{1}{N^2} \sum_{i,j=0}^{N-1} f_0(S^i x)f_1(T^j x)f_2(S^{i}T^{j}x)
 \end{equation}
 for commuting transformations $S$ and $T$.

The existence of the pointwise limit of (\ref{EQ1}) was proved by Assani \cite{A1} for three transformations and it was generalized to an arbitrary number of transformations by Chu and Frantzikinakis \cite{CF}. It is worth noting that in fact no assumption of commutativity of the transformations is required. 

In contrast, the average (\ref{EQ2}) has a very different nature. Leibman \cite{L} showed that convergence of (\ref{EQ2}) fails (even in $L^2$) without commutativity assumptions. When the transformations commute, the $L^2$ convergence of (\ref{EQ2}) (and its higher order versions) was first proved by Austin \cite{Aus} based on the work of Tao \cite{Tao} and then by Host \cite{H} using a different method. In order to prove this result, Host introduced the notion of magic extensions, which allows one to study such averages in an extension system with convenient properties. It is natural to ask if the averages in (\ref{EQ2}) converges in the pointwise sense. In this paper, we prove:


\begin{thm} \label{THM}
Let $(X,\mu,S,T)$ be an ergodic measure preserving system with commuting transformations $S$ and $T$. Then the average  \[\frac{1}{N^2} \sum_{i,j=0}^{N-1} f_0(S^i x)f_1(T^j x)f_2(S^i T^j x)\] converges a.e. as $N$ goes to infinity for any $f_0,f_1,f_2 \in L^{\infty}(\mu)$.
\end{thm}

Recently Huang, Shao and Ye \cite{HSY2} proved the pointwise convergence of multiple averages for a single transformation on a distal system. So a natural question arises from Theorem \ref{THM}: If $(X,\mu,S,T)$ is an ergodic measure preserving system with commuting transformations $S$ and $T$, does the average 
\[ \frac{1}{N} \sum_{i=0}^{N-1} f_0(S^ix)f_1(T^ix)\] 
converge in the pointwise sense as $N$ goes to infinity? Very little is known towards this question up to now. In \cite{DT}, Demeter and Thiele obtained the pointwise convergence a variation of this average. The case when $S$ and $T$ are powers of some ergodic transformation was solved by Bourgain \cite{Bo}, but no further results were known. 

\subsection{Strict ergodicity for dynamical cubes}
The main ingredient in proving Theorem \ref{THM} is to find a suitable topological model for the original system. This means finding a measurable conjugacy to a space with a convenient topological structure. Jewett-Krieger's Theorem states that every ergodic system has a strictly ergodic model (see Section 2.2 for definitions) and it is known that one can add some additional properties to the topological model.

In this paper, we are interested in the strict ergodicity property of the {\it dynamical cube} space of a topological model.
Let $X$ be a compact metric space and $S, T\colon X\to X$ be two commuting homeomorphisms. The dynamical cube space $\bold{Q}_{S,T}(X)$ is defined to be
$$\bold{Q}_{S,T}(X)=\overline{\{(x,S^ix,T^jx,S^iT^jx)\colon x\in X, i,j\in \Z\}}.$$
This object was introduced in \cite{DS} motived by Host's work \cite{H} and results in a useful tool to study products of minimal systems and their factors. A classical argument using Birkhoff Ergodic Theorem (see, for example, the proof of Theorem \ref{example}) shows that the strict ergodicity property of $\bold{Q}_{S,T}(X)$ is connected to pointwise multiple convergence problems such as Theorem \ref{THM} and Theorem \ref{example}.
We ask the following question:

\begin{ques}
  For any ergodic system $(X,\mu,S,T)$ with two commuting transformations $S$ and $T$, is there a topological model $(\wh{X},\wh{S},\wh{T})$ of $X$ such that $(\Q_{S,T}(\widehat{X}),\mathcal{G}_{\wh{S},\wh{T}})$ is strictly ergodic? Here $\mathcal{G}_{\wh{S},\wh{T}}$ is the group of action generated by $\id\times\wh{S}\times \id \times \wh{S}$, $\id\times\id\times \wh{T} \times \wh{T}$ and $\wh{R}\times \wh{R}\times \wh{R}\times \wh{R}$, where $\wh{R}=\wh{S}$ or $\wh{T}$.
\end{ques}

Huang, Shao and Ye \cite{HSY} gave an affirmative answer to this question for the case $S=T$. Although this question remains open in the general case, such a model always exists in an extension system of the original one. We prove the following theorem, which is the main tool to study Theorem \ref{THM}:
\begin{thm}\label{thm2}
  For any ergodic system $(X,\mu,S,T)$ with two commuting transformations $S$ and $T$, there exists an extension system $(Y,\nu,S,T)$ of $X$ and a topological model $(\wh{Y},\wh{S},\wh{T})$ of $Y$ such that $(\Q_{S,T}(\widehat{Y}),\mathcal{G}_{S,T})$ is strictly ergodic.
\end{thm}
It is worth noting that since every measurable function on the original system can be naturally lifted to a function on the extension system, this result is already sufficient for our purposes.

\subsection{Proof Strategy and organization}
Conventions and background material are in Section 2.
To prove Theorem \ref{thm2}, we refine the technique of Host in \cite{H} to find a suitable magic extension of the original system in Section 3. Then we use the method of Huang, Shao and Ye \cite{HSY} to find a desired model for this extension system in Section 4.
The announced pointwise convergence result (Theorem \ref{THM}) follows from Theorem \ref{thm2}, and
 we explain how this is achieved in Section 5.




\section{Background Material}
\subsection{Measure preserving systems}
A {\it measure preserving system} is a 4-tuple $(X,\mathcal{X},\mu,G_0)$, where $(X,\mathcal{X},\mu)$ is a probability space and $G_0$ is a group of measurable, measure preserving transformations acting on $X$.  When there is no confusion, we omit the $\sigma$-algebra $\mathcal{X}$ and assume without lose of generality that the probability space is standard.

A measure preserving system $(X,\mu,G_0)$ is {\it ergodic} if any $G_0$-invariant set of $X$ has measure 0 or 1.

If $T\colon X\to X$ is an invertible, measurable, measure preserving transformation, we let $(X,\mu,T)$ denote the measure preserving system $(X,\mu,\{T^i:i\in \Z\})$. If $S\colon X \to X$ and $T\colon X\to X$ are two commuting measure preserving transformations of $X$ (i.e $ST=TS$),  we write $(X,\mu,S,T)$ to denote the measure preserving system $(X,\mu,\{S^iT^j: i,j\in \Z \})$.

A {\it factor map} between the measure preserving systems $(Y,\nu,G_0)$ and $(X,\mu,$ $G_0)$ is a measure preserving map $\pi\colon Y \to X$ such that $\pi\circ g=g\circ\pi$ for all $g\in G_{0}$. If $\pi$ is a bi-measurable bijection, we say that $\pi$ is an {\it isomorphism} and that $(Y,\nu,G_0)$ and $(X,\mu,G_0)$ are {\it isomorphic}.



\subsection{Topological dynamical systems and models}

A {\it topological dynamical system} is a pair $({X},{G}_{0})$, where ${X}$ is a compact metric space and ${G}_{0}$ is a group of homeomorphisms of the space ${X}$.  A topological system $({X},G_0)$ is {\it minimal} if for any $x\in {X}$, its orbit $\{gx: g\in G_0\}$ is dense in ${X}$.

If ${S}\colon {X}\to {X}$ and ${T}\colon {X}\to{X}$ are two commuting homeomorphisms of ${X}$,  we write $({X},{T})$ to denote $({X},\{{T}^n:n\in \Z\})$  and $({X},{S},{T})$ to denote $({X},\{{S}^n{T}^m: n,m\in \Z \})$. Since we deal with both measure preserving systems and topological dynamical systems, we always write the measure for a measure preserving system to distinguish them.

\begin{conv}
 Throughout this paper, when we consider a system (measurable or topological) $(X,\mu,S,T)$ with commuting transformations $S$ and $T$, we always use $G\cong\mathbb{Z}^2$ to denote the group generated by $S$ and $T$.
\end{conv}

A (topological) {\it factor map} between the topological dynamical systems $(Y,G_{0})$ and $(X,G_{0})$ is an onto, continuous map $\pi\colon Y\to X$ such that $\pi\circ g=g\circ\pi$ for all $g\in G_{0}$. We say that $(Y,G_{0})$ is an {\it extension} of $(X,G_{0})$ or that $(X,G_{0})$ is a {\it factor} of $(Y,G_{0})$. When $\pi$ is bijective, we say that $\pi$ is an {\it (topological) isomorphism} and that $(Y,G_0)$ and $(X,G_0)$ are {\it (topological) isomorphic}.

By the Krylov-Bogolyubov Theorem, every topological dynamical system $({X},G_0)$ admits an invariant measure. When this measure is unique, we say that $({X},G_0)$ is {\it uniquely ergodic}. In addition, we say that $({X},G_0)$ is {\it strictly ergodic} if it is minimal and uniquely ergodic.

We state here a well known theorem for the case when $G_0$ is spanned by $d$ commuting transformations $T_1,\ldots, T_d$.

\begin{thm} Let $({X},G_0)$ be a topological dynamical system. The following are equivalent
\begin{enumerate}
\item $({X},G_0)$ is uniquely ergodic.

\item For any continuous function $f$, the average \[ \frac{1}{N^d} \sum_{i_1,\ldots,i_d \in [0,N-1]}  f(T_1^{i_1}\cdots T_d^{i_d}x) \]
converges uniformly to $\int f d\mu$ as $N$ goes to infinity.

\end{enumerate}

\end{thm}

 A deep connection between measure preserving systems and topological dynamical systems is the Jewett-Krieger Theorem \cite{J,K} which asserts that every ergodic system $(X,\mu,T)$ is isomorphic to a strictly ergodic topological dynamical system  $(\wh{X},\wh{\mu},\wh{T})$, where $\wh{\mu}$ is the unique ergodic measure of $(\wh{X},\wh{T})$. We say that $(\wh{X},\wh{T})$ is a {\it topological model} for $(X,\mu,T)$.

Further refinements have been given to the Jewett-Krieger Theorem. We state the one which is useful for our purposes.

\begin{defn}
 Let $(X,\mu,G_0)$ be a measure preserving system. We say that $G$ acts {\it freely} on $X$ (or the system $(X,\mu,G_0)$ is {\it free}) if any non-trivial $g\in G$ defines a transformation different from the identity transformation on $X$.
\end{defn}

Particularly, we say that a system $(X,\mu,S,T)$ with commuting transformations is {\it free} if $S^iT^j$ is not the identity transformation on $X$ for any $(i,j)\neq (0,0)$.

\begin{thm}[Weiss-Rosenthal \cite{W}] \label{WeissRosenthal}
Let $G_0$ be an amenable group and let $\pi\colon Y \to X$ be a factor map between two measure preserving systems $(Y,\nu,G_0)$ and $(X,\mu,G_0)$. Suppose that $(X,\mu,G_0)$ is free and $(\wh{X},\wh{G}_0)$ is a strictly ergodic model for $(X,\mu,G_0)$. Then there exits a strictly ergodic model $(\wh{Y}, G_0)$ for $(Y,\nu,G_0)$ and a topological factor map $\wh{\pi}:\wh{Y}\to \wh{X}$ such that the following diagram commutes:

\begin{figure}[h]
 \begin{tikzpicture}
  \matrix (m) [matrix of math nodes,row sep=3em,column sep=4em,minimum width=2em,ampersand replacement=\&]
  {
     X \& \widehat{X} \\
     Y\times W \& \widehat{Y}\times \widehat{W} \\};
  \path[-stealth]
     (m-1-1) edge node [left] {$\pi$} (m-2-1)
    (m-1-1) edge node [above] {$\Phi$} (m-1-2)
    (m-1-2) edge (m-1-1)
    (m-2-2) edge (m-2-1)
    (m-1-2) edge node [right] {$\widehat{\pi}$} (m-2-2)
    (m-2-1) edge node [below] {$\phi$} (m-2-2);
\end{tikzpicture}
 \end{figure}


Here we mean that $\Phi$ and $\phi$ are measure preserving isomorphisms and $\pi\circ \Phi=\phi \circ \widehat{\pi}$.

\end{thm}

In this case, we say that $\wh{\pi}\colon \wh{Y}\to \wh{X}$ is a {\it topological model} for $\pi\colon Y\to X$.




\subsection{Host magic extensions}
The Host magic extension was first introduced in \cite{H} to prove the $L^2$ convergence of multiple ergodic averages for systems with commuting transformations. Then Chu \cite{Chu} used this tool to study the recurrence problems in the same setting of systems. We recall that this construction is valid for an arbitrary number of transformations, but for convenience we state it only for two transformations $S$ and $T$.

\subsubsection{The Host measure}
\begin{defn}
  For any measure preserving transformation $R$ of the system $(X,\mathcal{X},\mu)$, we let $\mathcal{I}_R$ denote the $\sigma$-algebra of $R$-invariant sets.
\end{defn}

Let $X^{\ast}$ denote the space $X^4$. Let $\mu_S$ be the relative independent square of $\mu$ over $\mathcal{I}_S$,
meaning that for all $f_0,f_1\in L^{\infty}(\mu)$ we have $$ \int_{X^2} f_0(x_0)f_1(x_1)d\mu_S=\int_{X} \mathbb{E}(f_1|\mathcal{I}_S)\mathbb{E}(f_1|\mathcal{I}_S)d\mu,$$
where $\mathbb{E}(f|\mathcal{I}_S)$ is the {\it conditional expectation} of $f$ on $\mathcal{I}_S$.
It is obvious that $\mu_{S}$ is invariant under $\id \times S$ and $g\times g$ for $g\in G$.

Let $\mu_{S,T}$ denote the relative independent square of $\mu_{S}$ over $\mathcal{I}_{T\times T}$. Hence for all $f_0,f_1,f_2,f_3\in L^{\infty}(\mu)$ we have that
$$\int_{X^4}f_0(x_0)f_1(x_1)f_2(x_2)f_3(x_3)d\mu_{S,T}=\int_{X^2} \mathbb{E}(f_0\otimes f_1|\mathcal{I}_{T\times T})\mathbb{E}(f_2\otimes f_3|\mathcal{I}_{T\times T})d\mu_{S}.  $$
The measure $\mu_{S,T}$ is invariant under $\id\times S\times \id \times S$, $\id\times \id \times T \times T$ and under $g\times g\times g\times g $ for all $g\in G$.

Let ${S^{\ast}}$ and ${T^{\ast}}$ denote the transformations $\id\times S\times \id \times S$ and $\id\times \id \times T \times T$ respectively.
Then $(X^{\ast},\mu_{S,T}, S^{\ast},T^{\ast})$ is a system with commuting transformations $S^{\ast}$ and $T^{\ast}$. Let $\pi$ denote the projection $(x_0,x_1,x_2,x_3)\to x_3$ from ${X}^{\ast}$ to $X$. Then $\pi$ defines a factor map between $(X^{\ast},\mu_{S,T},  S^{\ast},T^{\ast})$ and $(X,\mu,S,T)$. We remark that the system $(X^{\ast},\mu_{S,T}, S^{\ast},T^{\ast})$ may not be ergodic even if $(X,\mu,S,T)$ is ergodic.

\subsubsection{The Host seminorm} Let $f\in L^{\infty}(\mu)$. The Host seminorm \cite{H} is defined to be the quantity

$$\normm{f}_{\mu,S,T}=\Bigl(\int_{X^4}  f(x_0)f(x_1)f(x_2)f(x_3)d\mu_{S,T}\Bigr)^{1/4}.$$

We have
\begin{prop}[\cite{H}, Proposition 2] \label{Cauchy}
\quad
\begin{enumerate}
\item For $f_0,f_1,f_2,f_3\in L^{\infty}(\mu)$, we have
 $$\int_{X^4} f_0\otimes f_1 \otimes f_2\otimes f_3 d\mu_{S,T}\leq \normm{f_0}_{\mu,S,T}\normm{f_1}_{\mu,S,T}\normm{f_2}_{\mu,S,T}\normm{f_3}_{\mu,S,T} $$

\item $\normm{\cdot}_{\mu,S,T}$ is a seminorm on $L^{\infty}(\mu)$.

\end{enumerate}
\end{prop}

We recall some standard notation.
For any two $\sigma$-algebras $\mathcal{A}$ and $\mathcal{B}$ of $X$, let $\mathcal{A}\vee\mathcal{B}$ denote the $\sigma$-algebra generated by $\{A\cap B\colon A\in \mathcal{A}, B\in \mathcal{B}\}$. If $f$ is a measurable function on $(X,\mathcal{X},\mu)$ and $\mathcal{A}$ is a sub-algebra of $\mathcal{X}$, let $\mathbb{E}(f|\mathcal{A})$ denote the conditional expectation of $f$ over $\mathcal{A}$.

\begin{defn}
 Let $(X,\mu,S,T)$ be a measure preserving system with commuting transformations $S$ and $T$. We say that $(X,\mu,S,T)$ is {\it magic} if
 $$\mathbb{E}(f|\mathcal{I}_S\vee \mathcal{I}_T)=0 \text{ if and only if } \normm{f}_{\mu,S,T}=0.$$

\end{defn}
The connection between the Host measure $\mu_{S,T}$ and magic systems is:
\begin{thm}[\cite{H}, Theorem 2]\label{mig}
The system $(X^{\ast},\mu_{S,T}, {S^{\ast}},{T^{\ast}})$ defined in Section 2.3.1 is a magic extension system of $(X,\mu,S,T)$.
\end{thm}

\subsection{Dynamical cubes}

The following notion of dynamical cubes for a system with commuting transformations was introduced and studied in \cite{DS}:
\begin{defn}
  Let $(X,S,T)$ be a topological dynamical system with commuting transformations $S$ and $T$. We let $\mathcal{G}_{S,T}$ denote the subgroup of $G^{4}$ generated by $\id\times S\times \id \times S$, $\id\times\id\times T \times T$ and $g\times g\times g\times g, g\in G$. For any $R\in G$, let $\mathcal{G}_{R}$ denote the subgroup of $G^{2}$ generated by $\id\times R$ and $g\times g, g\in G$.
\end{defn}

\begin{defn}
Let $(X,S,T)$ be a topological dynamical system with commuting transformations $S$ and $T$ and let $R\in G$. We define
\begin{equation}\nonumber
  \begin{split}
  &\bold{Q}_{S,T}(X)=\overline{\{(x,S^ix,T^jx,S^iT^jx)\colon x\in X, i,j\in \Z\}};
  \\&\bold{Q}_{R}(X)=\overline{\{(x,R^{i}x)\in X\colon x\in X, i\in\mathbb{Z}\}}.
  \end{split}
\end{equation}
\end{defn}

\section{The existence of free magic extensions}
In this section,
we strengthen Theorem \ref{mig} for our purposes by requiring the magic extension to be also ergodic and free.
We remark that there are a lot of interesting systems with commuting transformations where the action is not free. For example, the system $(X,\mu,S,S^i)$, where $S$ is an ergodic measure preserving transformation of $X$ and $i\in \Z, i\neq 1$.
However, we have

\begin{thm} \label{MagicFree}
Let $(X,\mu,S,T)$ be an ergodic system with commuting transformations $S$ and $T$. Suppose that $S^i$ and $T^j$ are not the identity for any $i,j\in \Z\setminus \{0\}$. Then there exists a magic extension $(\widehat{X},\nu, {S^{\ast
}},{T^{\ast}})$ where the action of $\Z^2$ is free and ergodic.
\end{thm}
\begin{rem} By Theorem \ref{mig}, $(X^{\ast},\mu_{S,T},{S^{\ast}
},{T^{\ast}})$ is a magic extension of $X$, but
  since $(X^{\ast},\mu_{S,T},{S^{\ast}
},{T^{\ast}})$ may not be ergodic, we need to decompose the measure $\mu_{S,T}$ in order to get an ergodic magic extension of $X$.
\end{rem}
\begin{proof}
Consider the measure $\mu_{S,T}$ on $X^{\ast}=X^4$. We claim that $\mu_{S,T}(\{\vec{x}: {S^{\ast}}^i{T^{\ast}}^j \vec{x}\neq \vec{x} \})=1$ for every $i,j\in \Z$. Let ${A^{\ast}}_{i,j}$ denote the set $\{\vec{x}: {S^{\ast}}^i{T^{\ast}}^j \vec{x}\neq \vec{x} \}$. Then the complement of ${A^{\ast}}_{i,j}$ is included in the union of the sets $X\times A\times X\times X$ and $X\times X\times B\times X$, where $A=\{x : S^i x=x\}$ and $B=\{x: T^jx=x\}$. Since the projection of $\mu_{S,T}$ onto any coordinate equals $\mu$, we have that $\mu_{S,T}({A^{\ast}}_{i,j}^c)\leq \mu(A)+\mu(B)=0$. Therefore, writing ${A^{\ast}}=\bigcap_{i,j\in \Z} {A^{\ast}}_{i,j}$, we have that $\mu_{S,T}({A^{\ast}})=1$.

Let $$\mu_{S,T}=\int \mu_{S,T,\vec{x}}d\mu_{S,T}(\vec{x})$$ be the ergodic decompositions of $\mu_{S,T}$ under ${S^{\ast}}$ and ${T^{\ast}}$.
Then we have that $\mu_{S,T,\vec{x}}({A^{\ast}})=1$ for $\mu_{S,T}$-a.e. $\vec{x}\in \widehat{X}$. By Proposition 3.13 of \cite{Chu}, for $\mu_{S,T}$-almost every $\vec{x}\in \widehat{X}$, the system $(X^{\ast},\mu_{S,T,\vec{x}}, {S^{\ast}},{T^{\ast}})$ is a magic extension of $(X,\mu,S,T)$. Hence, we can pick $\vec{x}_0\in A^{\ast}$ such that $(X^{\ast},\mu_{S,T,\vec{x}_0}, {S^{\ast}},{T^{\ast}})$ is a magic extension. This is a magic ergodic free extension of $(X,\mu,S,T)$.
\end{proof}




We prove some properties for later use.
In the rest of this section, we assume that $(X,\mu,S,T)$ is a free magic ergodic measure preserving system. Let $\W$ denote the $\sigma$-algebra $\mathcal{I}_S\vee \mathcal{I}_T$ and let $\mathcal{Z}_{S,T}$ be the factor associated to this $\sigma$-algebra.

\begin{lem} \label{ProductFactor}
 $\mathcal{Z}_{S,T}$ is isomorphic to the product of two ergodic systems. 
\end{lem}
\begin{proof}
Let $A\in \mathcal{I}_T$ and $B\in \mathcal{I}_S$. We have that $\lim_{N\to\infty}\frac{1}{N^2}\sum_{i,j=0}^{N-1} 1_{A}\circ S^i \circ T^j$ converges in $L^2(\mu)$ to $\mu(A)$. Since $A$ is invariant under $T$, we have that $\lim_{N\to\infty}\frac{1}{N} \sum_{j=0}^{N-1}1_{A}\circ S^j$ converges to $\mu(A)$. Similarly $\lim_{N\to\infty}\frac{1}{N} \sum_{j=0}^{N-1}1_{B}\circ S^j$ converges to $\mu(B)$. It follows that $$\lim_{N\to\infty}\frac{1}{N^2}\sum_{i,j=0}^{N-1} 1_{A\cap B}\circ S^i \circ T^j=\lim_{N\to\infty}\frac{1}{N^2}\sum_{i,j=0}^{N-1} 1_{A}1_{B}\circ S^i \circ T^j =\mu(A)\mu(B).  $$  Since $(X,\mu,S,T)$ is ergodic, this limit equals $\mu(A\cap B)$ and therefore $\mu(A\cap B)=\mu(A)\mu(B)$.

We conclude that the map $A\cap B\to A\times B$ defines a measure preserving isomorphism between $(X,\mathcal{I}_T\vee \mathcal{I}_S,\mu,S,T)$ and $(X\times X, \mathcal{I}_T\otimes \mathcal{I}_S,\mu\otimes \mu, S\times \id, \id \times T)$.

\end{proof}

For convenience, we write $(\mathcal{Z}_{S,T},S,T)=(Y\times W, \sigma\times \id,\id \times \tau)$.

\begin{lem} \label{Measurable}
 The $\sigma$-algebra of $(T\times T)$-invariant sets on $(X^2,\mu_S)$ is measurable with respect to $\W^2$.
\end{lem}

\begin{proof}
 We follow the proof of Proposition 4.7 of \cite{HK}.
 It suffices to show that $$\mathbb{E}(f_0\otimes f_1|\mathcal{I}(T\times T))=\mathbb{E}(\mathbb{E}(f_0|\W)\otimes\mathbb{E}(f_1|\W)|\mathcal{I}(T\times T)).$$

 It suffices to to prove this equality when $\mathbb{E}(f_i|\W)=0$ for $i=0$ or 1. By Proposition \ref{Cauchy}, we have that

 $$ \int f_0\otimes f_1\otimes f_0\otimes f_1 d\mu_{S,T}= \int |\mathbb{E}(f_0\otimes f_1|\mathcal{I}(T\times T)|^2d\mu_S \leq \normm{f_0}\normm{f_1}\normm{f_0}\normm{f_1},$$
 which implies that $\mathbb{E}(f_0\otimes f_1|\mathcal{I}(T\times T))=0$ whenever $\normm{f_i}=0$ for $i=0$ or 1. Since the system is magic, this is equivalent to $\mathbb{E}(f_i|\W)=0$ for $i=0$ or 1, and we are done.

\end{proof}


\section{Strict ergodicity for dynamical cubes} This section is devoted to the proof of Theorem \ref{thm2}. 
By Theorem \ref{MagicFree}, it suffices to prove the following theorem:

\begin{thm} \label{QStrictlyErgodic}
For any free ergodic magic system $(X,\mu,S,T)$ with two commuting transformations $S$ and $T$, there exists a topological model $(\wh{X},\wh{S},\wh{T})$ of $X$ such that $(\Q_{S,T}(\widehat{X}),\mathcal{G}_{S,T})$ is strictly ergodic.
\end{thm}


\subsection{A special case: product systems}
We start by proving a special case of Theorem \ref{QStrictlyErgodic}.
\begin{lem} \label{ErgodicProduct}
 Let $({Y},{\sigma})$ and $({W},{\tau})$ be two strictly ergodic systems with unique measures $\rho_{{Y}}$ and $\rho_{{W}}$. Then $({Y}\times {W},{\sigma}\times \id, \id\times {\tau})$ is strictly ergodic with measure $\rho_{{Y}}\otimes\rho_{{W}}$.
\end{lem}

\begin{proof}
 Let $\l$ be an invariant measure on ${Y}\times {W}$. Since ${Y}$ is uniquely ergodic, the projection onto the first coordinate of $\l$ is $\rho_{{Y}}$.
 Using the disintegration with respect to ${Y}$, we have that

 $$ \l=\int_{{Y}}\delta_{y}\times \lambda_y d\rho_{{Y}}.$$

 Since $\l$ is invariant under $\id\times \tau$, we have that $$(\id\times \tau)\lambda=\lambda=\int_{{Y}}\delta_{y}\times \tau\lambda_{y} d\rho_{{Y}}. $$

By the uniqueness of the disintegration, we get that $\tau\lambda_y=\lambda_y$ $\rho_{{Y}}$-a.e. Since $({W},\tau)$ is uniquely ergodic, a.e. we have that $\lambda_y=\rho_{{W}}$ and therefore
$$\lambda=\int_{{Y}}\delta_{y}\times \rho_{{W}}d\rho_{{Y}}=\rho_{{Y}}\otimes \rho_{{W}}. $$
\end{proof}

The next corollary follows immediately from Lemma \ref{ErgodicProduct}:

\begin{cor} \label{Products}
 Let $((X_i,T_i))_{i=1}^n$ be strictly ergodic systems with measures $(\rho_i)_{i=1}^n$. The system $(\prod X_i,\otimes T_i)$ is strictly ergodic with measure $\otimes \rho_i.$
\end{cor}

We are now ready to prove Theorem \ref{QStrictlyErgodic} for the case when the system is a product:

\begin{prop} \label{Q.u.e}
 Let $({Y},{\sigma})$ and $({W},{\tau})$ be two strictly ergodic systems with unique measures $\rho_{{Y}}$ and $\rho_{{W}}$. Then $\Q_{{\sigma}\times \id, \id \times {\tau}}({Y}\times {W})$ is uniquely ergodic with measure $\nu_{{\sigma}\times \id, \id \times {\tau}}$, where $\nu=\rho_{{Y}}\otimes \rho_{{W}}$. Particularly, $(\Q_{\sigma\times \id}(Y\times W),\mathcal{G}_{\sigma\times\id})$ is strictly ergodic with measure $\nu_{\sigma\times \id}$.
 \end{prop}

\begin{proof}
 By definition, we deduce that $$\Q_{{\sigma}\times \id, \id \times {\tau}}({Y}\times {W})=\left \{ \left ((y,w),(y',w),(y,w'),(y',w')\right ): y,y'\in {Y}, w,w'\in {W}\right \}$$

  and $\mathcal{G}_{\sigma\times \id,\id\times \tau}$ is the group spanned by
  \begin{align*}
   (\sigma\times \id)\times (\sigma\times \id)\times  (\sigma\times \id)\times (\sigma\times \id)\\
   (\id \times \tau) \times (\id \times \tau) \times (\id \times \tau) \times (\id \times \tau)\\
   (\id\times \id)\times (\sigma\times \id) \times (\id\times \id) \times (\sigma\times \id)\\
   (\id\times \id)\times (\id\times \id)\times (\id\times \tau)\times (\id\times \tau).
  \end{align*}


  We may identify $\Q_{{\sigma}\times \id, \id \times {\tau}}({Y}\times {W})$ with ${Y}\times {Y}\times {W}\times {W}$ under the map $\phi$ $$\left ((y,w),(y',w),(y,w'),(y',w')\right )\mapsto (y,y',w,w').$$
 We remark that $\mathcal{G}_{\sigma\times \id,\id\times \tau}$ is mapped to the group spanned by $$\sigma\times \sigma\times \id \times \id,\quad \id \times \id \times \tau \times \tau, \quad \id\times \sigma \times \id \times \id \quad \text{ and } \quad \id\times \id \times \id \times \tau.$$ This is the same as the group spanned by $$\sigma\times \id\times \id \times \id, \quad  \id \times \id \times \tau \times \id, \quad  \id\times \sigma \times \id \times \id \quad \text{  and   } \quad  \id\times \id \times \id \times \tau.$$ By Corollary \ref{Products}, this system is uniquely ergodic with measure $\rho_{{Y}}\otimes\rho_{{Y}}\otimes \rho_{{W}}\otimes\rho_{{W}}$. Since $\nu_{\sigma\times \id,\id \times \tau}$ is an invariant measure on $\Q_{{\sigma}\times \id, \id \times {\tau}}({Y}\times {W})$, we have that it is the unique invariant measure and it coincides with $\phi^{-1} (\rho_{{Y}}\otimes\rho_{{Y}}\otimes \rho_{{W}}\otimes\rho_{{W}})$.




\end{proof}

\subsection{Proof of the general case}
Throughout this section, we consider $(X,\mu,$ $S,T)$ as a fixed system which is magic, ergodic and free, and we follow the notations in the previous section.
By Lemma \ref{ProductFactor}, the factor associated to the $\sigma$-algebra $\W=\mathcal{I}_S\vee \mathcal{I}_T$ has the form $(Y\times W, \sigma\times \id,\id \times \tau)$, where $(Y,\sigma)$ and $(W,\tau)$ are ergodic systems.

\begin{lem} \label{TopologicalModel}
 There exists a strictly ergodic topological model for the factor map $\pi\colon X\to Y\times W$.
\end{lem}
\begin{proof}
 By the Jewett-Krieger Theorem, we can find strictly ergodic models $(\widehat{Y},\widehat{\sigma})$ and $(\widehat{W},\widehat{\tau})$ for $(Y,\sigma)$ and $(W,\tau)$, respectively. Let $\rho_{{Y}}$ and $\rho_{{W}}$ denote the unique ergodic measures on these systems. By Lemma \ref{ErgodicProduct}, $(\widehat{Y}\times \widehat{W},\widehat{\sigma}\times \id,\id\times\widehat{\tau})$ is a strictly ergodic model for $(Y\times W,\sigma\times \id,\id\times \tau)$ with unique invariant measure $\rho_{{Y}}\otimes \rho_{{W}}$.

By Theorem \ref{WeissRosenthal}, there exists a strictly ergodic model $\widehat{\pi}:\widehat{X}\to \widehat{Y}\times \widehat{W}$ for $\pi:X\to Y\times W$. 
\end{proof}

We are now ready to prove Theorem \ref{QStrictlyErgodic}:

\begin{proof}[Proof of Theorem \ref{QStrictlyErgodic}.]
For any free ergodic magic system $(X,S,T)$, let $\pi:X\to (Y\times W,\sigma\times \id,\id\times \tau)$ be the factor map associated to the $\sigma$-algebra $\W=\mathcal{I}_S\vee \mathcal{I}_T$. Let $\widehat{\pi}\colon \wh{X}\to \wh{Y}\times \wh{W}$ be the topological model given by Lemma \ref{TopologicalModel}. We claim that $(\Q_{S,T}(\widehat{X}),\mathcal{G}_{S,T})$ is strictly ergodic.

To simplify the notation, we replace $\wh{X}$, $\wh{W}$, $\wh{Y}$, etc by $X$, $W$, $Y$ etc.
It was proved in Proposition 3.14 of \cite{DS} that $(\Q_{S,T}({X}),\mathcal{G}_{S,T})$ is a minimal system. So it suffices to show unique ergodicity.

{{\bf Claim 1:} $(\Q_{S}(X),\mathcal{G}_S)$ is uniquely ergodic with measure $\mu_S$.}

We recall that the factor of $X$ corresponding to $\mathcal{I}_S$ is $(W,\id,\tau)$.


Suppose that the ergodic decomposition of $\mu$ under $S$ is
    \begin{equation}\nonumber
    \begin{split}
      \mu=\int_{W}\mu_{\omega} d\rho_{{W}}(\omega).
    \end{split}
  \end{equation}
Then
    \begin{equation}\nonumber
    \begin{split}
      \mu_{S}=\int_{W}\mu_{\omega}\times\mu_{\omega} d\rho_{{W}}(\omega).
    \end{split}
  \end{equation}

  Let $\pi_W \colon X\to W $ be the factor map and let $\l$ be a $\mathcal{G}_{S}$-invariant measure on $\Q_{S}({X})$. For $i=0,1$, let $p_{i}\colon (\Q_{S}({X}),\mathcal{G}_{S})\to (X,G)$ be the projection onto the $i$-th coordinate. Then $p_{i}\l$ is a $G$-invariant measure of $X$. Therefore, $p_{i}\l=\mu$. Hence we may assume that
  \begin{equation}\nonumber
    \begin{split}
      \l=\int_{X}\d_{x}\times\l_{x}d\mu(x)
    \end{split}
  \end{equation}
  is the disintegration of $\l$ over $\mu$. Since $\l$ is $(\id\times S)$-invariant, we have that
    \begin{equation}\nonumber
    \begin{split}
      \l=(\id\times S)\l=\int_{X}\d_{x}\times\l_{Sx}d\mu(x).
    \end{split}
  \end{equation}
  The uniqueness of disintegration implies that $\l_{Sx}=\l_{x}$ for $\mu$-a.e. $x\in X$. So the map
  $$F\colon X\to M(X)\colon x\mapsto \l_{x}$$
  is an $S$-invariant function. Hence we can write $\l_{x}=\l_{\pi_{W}(x)}$ for $\mu$-a.e. $x\in X$.

  Then we have
    \begin{equation}\nonumber
    \begin{split}
      &\l=\int_{X}\d_{x}\times\l_{x}d\mu(x)=\int_{X}\d_{x}\times\l_{\pi_W(x)}d\mu(x)
      \\&\qquad=\int_{W}\int_{X}\d_{x}\times\l_{\omega}d\mu_{\omega}(x)d\rho_{{W}}(\omega)
      \\&\qquad=\int_{W}(\int_{X}\d_{x}d\mu_{\omega}(x))\times\l_{\omega}d\rho_{{W}}(\omega)
      \\&\qquad=\int_{W}\mu_{\omega}\times\l_{\omega}d\rho_{{W}}(\omega).
    \end{split}
  \end{equation}

  Recall that $\Q_{\id}(W)=\Delta_W$ and $\mathcal{G}_{\id}$ is spanned by $(\tau,\tau)$. Therefore $(\Q_{\id}(W),$ $\mathcal{G}_{\id})$ is isomorphic to $(W,\tau)$. Particularly, it is uniquely ergodic and for convenience we let $P_W$ denote its invariant measure.

  Let $\pi_Y^{2}\colon (\Q_{S}(X), \mathcal{G}_{S})\to(\Q_{\id}(W), \mathcal{G}_{\id})$ be the natural factor map.
  We have that
  $$\pi_W^2(\l)=P_W.$$
  Thus
 \begin{equation}\label{temp1}
    \begin{split}
    \pi(\mu_{\omega})=\pi(\l_{\omega})=\delta_{\omega}.
    \end{split}
  \end{equation}
  On the other hand, $p_{1}(\l)=p_{2}(\l)=\mu$ implies that
      \begin{equation}\label{temp2}
    \begin{split}
\mu=\int_{W}\mu_{\omega}d\rho_W(\omega)=\int_{W}\l_{\omega}d\rho_{W}(\omega).
    \end{split}
  \end{equation}
By (\ref{temp1}), (\ref{temp2}) and the uniqueness of disintegration, we have that $\l_{\omega}=\mu_{\omega}, \rho_{{W}}$-a.e. $\omega\in W$. So
$$\l=\int_{W}\mu_{\omega}\times\mu_{\omega}d\rho_{{W}}(\omega)=\mu_{S}.$$

This finishes the proof of Claim 1.

{{\bf Claim 2}: $(\Q_{S,T}({X}),\mathcal{G}_{S,T})$ is uniquely ergodic with unique measure $\mu_{S,T}$.}

Let $\l$ be a $\mathcal{G}_{S,T}$-invariant measure on $\Q_{S,T}({X})$. Let $\overline{p}_{1}, \overline{p}_{2}\colon (\Q_{S,T}({X}),\mathcal{G}_{S,T})$ $\to (\Q_{S}({X}),\mathcal{G}_S)$ be the projection onto the first two and last two coordinates, respectively. Then $\overline{p}_{i}\l$ is a $\mathcal{G}_{S}$-invariant measure of $\Q_{S}({X})$ and therefore, $\overline{p}_{i}\l=\mu_{S}$. Hence we may assume that
  \begin{equation}\nonumber
    \begin{split}
      \l=\int_{\Q_{S}(X)}\delta_{\x}\times\l_{\x}d\mu_{S}(\x)
    \end{split}
  \end{equation}
  is the disintegration of $\l$ over $\mu_{S}$. Since $\l$ is $(\id\times\id\times T\times T)$-invariant, we have that

\begin{equation}\nonumber
    \begin{split}
      \l=(\id\times \id \times T\times T)\l=\int_{\Q_{S}(X)}\d_{\x}\times\l_{(T\times T)\x}d\mu_{S}(\x).
    \end{split}
  \end{equation}
  The uniqueness of disintegration implies that $\l_{(T\times T)\x}=\l_{\x}$ for $\mu_{S}$-a.e. $\x\in \Q_{S}(X)$. So the map
  $$F\colon \Q_{S}(X)\to M(X^{4})\colon \x\mapsto \l_{\x}$$
  is a $(T\times T)$-invariant function and therefore $F$ is $\I_{T\times T}$-measurable.

  Let $(\Omega_{S,T},\overline{P})$ be the factor of $(X\times X,\mu_S)$ corresponding to the subalgebra $\I_{T\times T}$ and let $\phi$ denote the corresponding factor map.
  Suppose that
the ergodic decomposition of $\mu_{S}$ under $T\times T$ is
    \begin{equation}\nonumber
    \begin{split}
      \mu_{S}=\int_{\Omega_{S,T}}\mu_{S,\omega}d\overline{P}(\omega).
    \end{split}
  \end{equation}
Then
    \begin{equation}\nonumber
    \begin{split}
      \mu_{S,T}=\int_{\Omega_{S,T}}\mu_{S,\omega}\times \mu_{S,\omega}d\overline{P}(\omega).
    \end{split}
  \end{equation}

  Hence we can write $\l_{\x}=\l_{\phi(\x)}$ for $\mu_{S}$-a.e. $\x\in \Q_{S}(X)$. 
  Then we have
    \begin{equation}\nonumber
    \begin{split}
      &\l=\int_{\Q_{S}(X)}\d_{\x}\times\l_{\x}d\mu_{S}(\x)=\int_{\Q_{S}(X)}\d_{\x}\times\l_{\phi(\x)}d\mu_{S}(\x)
      \\&\qquad=\int_{\Omega_{S,T}}\int_{\Q_{S}(X)}\d_{\x}\times\l_{\omega}d\mu_{S,\omega}(\x)d\overline{P}(\omega)
      \\&\qquad=\int_{\Omega_{S,T}}(\int_{\Q_{S}(X)}\d_{\x}d\mu_{S,\omega}(\x))\times\l_{\omega}d\overline{P}(\omega)
      \\&\qquad=\int_{\Omega_{S,T}}\mu_{S,\omega}\times\l_{\omega}d\overline{P}(\omega).
    \end{split}
  \end{equation}

  Recall that $\pi\colon X\to Y\times W$ is the factor map. Let $$\pi^{4}\colon (\Q_{S,T}(X), \mathcal{G}_{S,T})\to(\Q_{\sigma\times \id,\id\times \tau}(Y\times W), \mathcal{G}_{\sigma\times\id,\id\times \tau})$$ be the natural factor map. By Lemma \ref{Measurable}, there exists a factor map $\alpha\colon (Y\times W)^2\to \Omega_{S,T}$ such that $\alpha \circ \pi^2=\phi^2$. 

  Let $\nu=\rho_{{Y}}\otimes \rho_{{W}}$ denote the unique invariant measure on $Y\times W$. By Proposition \ref{Q.u.e}, we have that $(\Q_{S,T}(Y\times W), G_{\sigma\times \id,\id\times \tau})$ is uniquely ergodic and $\nu_{S,T}$ is its unique invariant measure.

  Suppose that the ergodic decomposition of $\nu_{S}$ under $T\times T$ is
  $$\nu_{S}=\int_{\Omega_{S,T}}\nu_{S,\omega}d\overline{P}(\omega).$$
  Then we have
    $$\nu_{S,T}=\int_{\Omega_{S,T}}\nu_{S,\omega}\times\nu_{S,\omega}d\overline{P}(\omega).$$
  Since $\pi^{4}\l$ is an invariant measure on $\Q_{\sigma\times\id,\id\times \tau}(Y\times W)$, we have that
  $$\pi^{4}(\l)=\nu_{S,T}=\int_{\Omega_{S,T}}\nu_{S,\omega}\times\nu_{S,\omega}d\overline{P}(\omega).$$
  Since $\phi^2=\alpha\circ\pi^{2}$, we have that
 \begin{equation}\label{temp11}
    \begin{split}
    \phi^2(\mu_{S,\omega})=\phi^2(\l_{\omega})=\alpha(\nu_{S,\omega})=\delta_{\omega}.
    \end{split}
  \end{equation}
  On the other hand, $\overline{p_{1}}(\l)=\overline{p}_{2}(\l)=\mu$ implies that
      \begin{equation}\label{temp21}
    \begin{split}
\mu_{S}=\int_{\Omega_{S,T}}\mu_{S,\omega}d\overline{P}(\omega)=\int_{\Omega_{S,T}}\l_{\omega}d\overline{P}(\omega).
    \end{split}
  \end{equation}
By (\ref{temp11}), (\ref{temp21}) and the uniqueness of disintegration, we have that $\l_{\omega}=\mu_{S,\omega}, \overline{P}$-a.e. $\omega\in\Omega_{S,T}$. So
$$\l=\int_{\Omega_{S,T}}\mu_{S,\omega}\times\mu_{S,\omega}d\overline{P}(\omega)=\mu_{S,T}.$$
Thus $(\Q_{S,T}({X}), \mathcal{G}_{S,T})$ is strictly ergodic with unique measure $\mu_{S,T}$.
\end{proof}

\section{Applications to pointwise results}


We apply results in previous sections to deduce some convergence results. We remark that if $S^i$ is the identity for some $i\neq 0$, the averages we consider in this section reduce to the Birkhoff ergodic theorem.
So the difficult case is when the systems $(X,\mu,S)$ and $(X,\mu,T)$ are free, and we make this assumption throughout this section. Since the averages we consider can be deduced by proving them in an extension of $X$, by Theorem \ref{MagicFree} we may assume that $(X,\mu,S,T)$ is a magic free ergodic system. By Theorem  \ref{QStrictlyErgodic}, we may take a strictly topological model $(\widehat{X},\widehat{S},\widehat{T})$ for $X$ such that $(\Q_{S,T}(\widehat{X}), \mathcal{G}_{\widehat{S},\widehat{T}})$ is strictly ergodic. So (omitting the symbol $\wh{\quad}$ to simplify notation),  throughout all this section we assume that $(X,\mu,S,T)$ is a magic free ergodic system and that $(\Q_{S,T}({X}), \mathcal{G}_{{S},{T}})$ is strictly ergodic.

%

\begin{thm}\label{example}
 Let $(X,\mu,S,T)$ be an ergodic measure preserving system. Let $f_0,f_1,f_2,f_3 \in L^{\infty}(\mu)$. Then

 $$ \lim_{N\to\infty}\frac{1}{N^4} \sum_{i,j,k,p=0}^{N-1} f_0(S^iT^jx)f_1(S^{i+k}T^jx)f_2(S^{i}T^{j+p}x)f_3(S^{i+k}T^{j+p}x)$$
 converges almost everywhere to $\int f_0\otimes f_1 \otimes f_2 \otimes f_3 d\mu_{S,T}$.
\end{thm}

\begin{proof}
 Since it suffices to prove the result in any extension system of $X$, by Theorem \ref{MagicFree}, we may assume that $(X,\mu,S,T)$ is a magic free ergodic system. By Theorem  \ref{QStrictlyErgodic}, we may take a strictly topological model $(\widehat{X},\widehat{S},\widehat{T})$ for $X$ such that $(\Q_{S,T}(\widehat{X}), \mathcal{G}_{\widehat{S},\widehat{T}})$ is strictly ergodic. To simplify the notation, we omit the symbol $\wh{\quad}$ in the sequel.

Recall that $\mathcal{G}_{S,T}$ is a $\Z^4$-action spanned by $S\times S\times S\times S$, $T\times T\times T\times T$, $\id\times S\times \id\times S$ and $\id\times \id \times T\times T$.

Let $f_0, f_1, f_2, f_3 \in L^{\infty}(\mu)$ and fix $\epsilon>0$. Let $\wh{f_0}, \wh{f_1},\wh{f_2},\wh{f_3}$ be continuous functions on $X$ such that $\|f_i-\wh{f}_i\|_{1}<\epsilon$ for $i=0,1,2,3$. We can assume that all functions are bounded by 1 in $L^{\infty}$ norm. For simplicity, denote $$I(h_0,h_1,h_2,h_3)=\int h_0\otimes h_1 \otimes h_2 \otimes h_3 d\mu_{S,T}$$ and $$\mathbb{E}_{N}(h_0\otimes h_1\otimes h_2\otimes h_3)(x)=\frac{1}{N^4} \sum_{i,j,k,p=0}^{N-1} h_0(S^iT^jx)h_1(S^{i+k}T^jx)h_2(S^{i}T^{j+p}x)h_3(S^{i+k}T^{j+p}x) $$
for $x\in X, h_0, h_1, h_2, h_3 \in L^{\infty}(\mu)$.
By the telescoping inequality, we have

\begin{align*}
&\qquad \left | \mathbb{E}_N(f_0\otimes f_1 \otimes f_2 \otimes f_3)(x) - I(f_0,f_1,f_2,f_3)   \right |   \\
& \leq  \left | \mathbb{E}_N(f_0\otimes f_1 \otimes f_2 \otimes f_3)(x) - \mathbb{E}_N(\wh{f}_0\otimes \wh{f}_1 \otimes \wh{f}_2 \otimes \wh{f}_3)(x)   \right | \\
& + \left |\mathbb{E}_N(\wh{f}_0\otimes \wh{f}_1 \otimes \wh{f}_2 \otimes \wh{f}_3)(x)- I(f_0,f_1,f_2,f_3) \right | \\
& \leq \frac{1}{N^2} \sum_{i,j}|f_0(S^iT^jx)-\widehat{f}_0(S^iT^jx)|+ \frac{1}{N^3} \sum_{i,j,k}|f_1(S^{i+k}T^jx)-\widehat{f}_1(S^{i+k}T^jx)|  \\
& +\frac{1}{N^3} \sum_{i,j,p}|f_2(S^iT^{j+p}x)-\widehat{f}_2(S^iT^{j+p}x)| + \frac{1}{N^4} \sum_{i,j,k,p}|f_3(S^{i+k}T^{j+p}x)-\widehat{f}_3(S^{i+k}T^{j+p}x)|\\
& + \left |\mathbb{E}_N(\wh{f}_0\otimes \wh{f}_1 \otimes \wh{f}_2 \otimes \wh{f}_3)(x)- I(\wh{f}_0,\wh{f}_1,\wh{f}_2,\wh{f}_3) \right |+ \left |I(f_0,f_1,f_2,f_3)- I(\wh{f}_0,\wh{f}_1,\wh{f}_2,\wh{f}_3) \right |.
\end{align*}

 Since $(\Q_{S,T}({X}), \mathcal{G}_{S,T})$ is uniquely ergodic, we have that \[\left |\mathbb{E}_N(\wh{f}_0\otimes \wh{f}_1 \otimes \wh{f}_2 \otimes \wh{f}_3)(x)- I(\wh{f}_0,\wh{f}_1,\wh{f}_2,\wh{f}_3) \right |\] converges to 0 for every $x\in X$ as $N$ goes to infinity.

On the other hand, by Birkhoff ergodic theorem, we have that the four first terms of the last inequality converge a.e. to $\|f_0-\wh{f}_0\|_1$, $\|f_1-\wh{f}_1\|_1$, $\|f_2-\wh{f}_2\|_1$ and $\|f_3-\wh{f}_3\|_1$, respectively.

Finally, by the telescoping inequality and the fact that the marginals of $\mu_{S,T}$ are equal to $\mu$ we deduce that $$\left |I(f_0,f_1,f_2,f_3)- I(\wh{f}_0,\wh{f}_1,\wh{f}_2,\wh{f}_3) \right |\leq \|f_0-\wh{f}_0\|_1+ \|f_1-\wh{f}_1\|_1+\|f_2-\wh{f}_2\|_1 + \|f_3-\wh{f}_3\|_1 .$$

Therefore, we can find $N$ large enough and a subset $X_N\subset X$ with measure larger than $1-\epsilon$ such that for every $x\in X_N$,
$$\left | \mathbb{E}_N(f_0\otimes f_1 \otimes f_2 \otimes f_3)(x) - I(f_0,f_1,f_2,f_3)   \right | \leq 13 \epsilon. $$

Since $\epsilon$ is arbitrary, we conclude that $\mathbb{E}_N(f_0\otimes f_1 \otimes f_2 \otimes f_3)$ converges to $I(f_{0},f_{1},f_{2},f_{3})$ a.e.
as $N$ goes to infinity.

\end{proof}

Since $(\Q_{S,T}(X),\mathcal{G}_{S,T})$ is uniquely ergodic, we also have:

\begin{lem}\label{ConvergenceContinuous}
Let $\wh{f}_0$, $\wh{f}_1$, $\wh{f}_2$, $\wh{f}_3$ be continuous functions on $X$. Then $$ \frac{1}{N^4} \sum_{i,j=0}^{N-1}\sum_{k=-i}^{N-1-i} \sum_{p=-j}^{N-1-j} f_0(S^iT^jx)f_1(S^{i+k}T^jx)f_2(S^{i}T^{j+p}x)f_3(S^{i+k}T^{j+p}x)  $$
converges to $I(\wh{f}_0,\wh{f}_1,\wh{f}_2,\wh{f}_3)$.
\end{lem}

\begin{proof}
Suppose that the averages does not converge to $I(\wh{f}_0,\wh{f}_1,\wh{f}_2,\wh{f}_3)$. Then there exist $x\in X$, a sequence $N_m\to \infty$ and $\epsilon>0$ such that the $N_m$-average at $x$ and the integral differs at least $\epsilon$. Take any weak$^{\ast}$-limit of the sequence $$\frac{1}{N^4}\sum_{i,j=0}^{N_m-1}\sum_{k=-i}^{N_m-1-i} \sum_{p=-j}^{N_m-1-j}(S^iT^j\times S^{i+k}T^j\times S^{i}T^{j+p}\times S^{i+k}T^{j+p})\delta_{(x,x,x,x)}.$$ Such a limit is clearly invariant under $\mathcal{G}_{S,T}$ and therefore it equals to $\mu_{S,T}$ by unique ergodicity. Hence,
\[ \frac{1}{N_m^4} \sum_{i,j=0}^{N_m-1}\sum_{k=-i}^{N_m-1-i} \sum_{p=-j}^{N_m-1-j} f_0(S^iT^jx)f_1(S^{i+k}T^jx)f_2(S^{i}T^{j+p}x)f_3(S^{i+k}T^{j+p}x)\] converges to $I(\wh{f}_0,\wh{f}_1,\wh{f}_2,\wh{f}_3)$ as $m$ goes to infinity, a contradiction.
\end{proof}



For any $N\in\N$, denote
$$A_{N}\coloneqq \{(i,j,k,p)\in\Z^{4}\colon i,k\in[0,N-1],k\in[-i,N-1-i],p\in[-j,N-1-j]\}.$$
 Let $(X,\mu,S,T)$ be a measure preserving system with commuting transformations $S$ and $T$. For any  $f\in L^{\infty}(X)$ and any $x\in X$, denote
\[S_{N}(f,x)\coloneqq \Bigl\vert\frac{1}{N^{4}}\sum_{(i,j,k,p)\in A_{N}}f(S^{i}T^{j}x)\overline{f(S^{i+k}T^{j}x)}\overline{f(S^{i}T^{j+p}x)}f(S^{i+k}T^{j+p}x)\Bigr\vert.\]

\begin{lem} \label{BoundAverage}
  Let $(X,\mu,S,T)$ be a measure preserving system with commuting transformations $S$ and $T$ and let $f_{1},f_{2},f_{3}\in L^{\infty}(X)$ with $\Vert f_{i}\Vert_{\infty}\leq 1, i=1,2,3$. Then there exists a universal constant $C$, such that for any $x\in X$ and any $N\in\N$, we have that
  $$\Bigl(\frac{1}{N^{2}}\sum_{i=0}^{N-1}\sum_{j=0}^{N-1}f_{1}(S^{i}x)f_{2}(T^{j}x)f_{3}(S^{i}T^{j}x)\Bigr)^{4}\leq C \vert S_{N}(f_{3},x)\vert.$$
\end{lem}
\begin{proof}
  By Cauchy-Schwartz inequality and the boundedness of $f_{1}$, the expression inside the parenthesis on the left hand side is bounded by a multiple of the square of
  \begin{equation}\label{Btemp1}
    \begin{split}
\frac{1}{N}\sum_{i=0}^{N-1}\Bigl(\frac{1}{N}\sum_{j=0}^{N-1}f_{2}(T^{j}x)f_{3}(S^{i}T^{j}x)\Bigr)^{2}
= \\ \frac{1}{N^{3}}\sum_{j=0}^{N-1}\sum_{h=-j}^{N-1-j}\sum_{i=0}^{N-1}f_{2}(T^{j}x)\overline{f_{2}(T^{j+p}x)}f_{3}(S^{i}T^{j}x)\overline{f_{3}(S^{i}T^{j+p}x)}.
    \end{split}
  \end{equation}
   By Cauchy-Schwartz inequality and the boundedness of $f_{2}$, the square of (\ref{Btemp1}) is bounded by a multiple of
     \begin{equation}\nonumber
    \begin{split}
      &\qquad\frac{1}{N}\sum_{j=0}^{N-1}\frac{1}{N}\sum_{p=-j}^{N-1-j}\Bigl(\frac{1}{N}\sum_{i=0}^{N-1}f_{3}(S^{i}T^{j}x)\overline{f_{3}(S^{i}T^{j+p}x)}\Bigr)^{2}
      \\&=\frac{1}{N}\sum_{j=0}^{N-1}\frac{1}{N}\sum_{h=-j}^{N-1-j}\frac{1}{N}\sum_{i=0}^{N-1}\frac{1}{N}\sum_{k=-i}^{N-1-i}f_{3}(S^{i}T^{j}x)\overline{f_{3}(S^{i}T^{j+p}x)}\overline{f_{3}(S^{i+k}T^{j}x)}f_{3}(S^{i+p}T^{j+p}x)
      \\&=S_{N}(f_{3},x).
    \end{split}
  \end{equation}
\end{proof}

Now we are able to prove the main result:

\begin{thm*}
Let $f_1,f_2,f_3 \in L^{\infty}(\mu)$. Then $$\lim_{N\to\infty}\frac{1}{N^2}\sum_{i,j=0}^{N-1} f_1(S^ix)f_2(T^jx)f_3(S^iT^jx)$$ converges a.e.

\end{thm*}

\begin{proof} We may assume without loss of generality that all the functions are bounded by 1 in $L^{\infty}$ norm.
Suppose first that $f_3=h_3h_3'$, where $h_3$ is measurable with respect to $\mathcal{I}_S$ and $h_3'$ is measurable with respect to $\mathcal{I}_T$. In this case, we have that $f_3(S^iT^jx)=h_3(S^ix)h_3'(T^jx)$ and thus $$\frac{1}{N^2}\sum_{i,j=0}^{N-1} f_1(S^ix)f_2(T^jx)f_3(S^iT^jx)= \frac{1}{N^2}\sum_{i,j=0}^{N-1} f_1(S^ix)h_3(S^ix)f_2(T^jx)h_3'(T^jx),$$
and so the average converges by Birkhoff Theorem. Therefore the average converges a.e. for any $f_3$ in the subspace $L$ spanned by those kind of functions. 
Any function $f_3$ measurable with respect to $\W$ can be approximated in the $L^1$ norm by functions in $L$. So, for $f_3$ measurable with respect to $\W$ we can take a sequence $(g_k)_{k\in \N}$ in $L$ that converge to $f_3$ in $L^1$ norm. 
By Birkhoff Theorem, there exists a set $A$ of full measure such that  
$$\limsup_{N\to\infty} \left | \frac{1}{N^2}\sum_{i,j=0}^{N-1} f_1(S^ix)f_2(T^jx)(f_3(S^iT^jx)-g_k(S^iT^jx))\right |\leq \| f_3-g_k \|_1 $$
for every $x\in A$ and $k\in \N$. Again by Birkhoff Theorem, let $B$ be a set of full measure such that the average 
$$\frac{1}{N^2}\sum_{i,j=0}^{N-1}f_1(S^ix)f_2(T^jx)g_k(S^iT^jx)$$
converges for all $x\in B$ and all $k\in\N$. It is easy to check that for  $x\in A\cap B$, the sequence $A_{N}=\frac{1}{N^2}\sum_{i,j=0}^{N-1} f_1(S^ix)f_2(T^jx)f_3(S^iT^jx)$ forms a Cauchy sequence and therefore it converges. 

We then suppose that $\mathbb{E}(f_3|\W)=0$. Let $\epsilon>0$ and let $\wh{f}_3$ be a continuous function on $X$ such that $\|f_3-\wh{f}_3\|_1<\epsilon$.
We have that
\begin{equation} \label{Bound}
\left | \frac{1}{N^2}\sum_{i,j=0}^{N-1} f_1(S^ix)f_2(T^jx)(f_3(S^iT^jx)-\wh{f}_3(S^iT^jx))\right |\leq  \frac{1}{N^2}\sum_{i,j=0}^{N-1} \left | f_3(S^iT^jx)-\wh{f}_3(S^iT^jx) \right |.
\end{equation}
By Birkhoff Theorem, the right hand side converges a.e. to $\|f_3-\wh{f}_3\|_1$ as $N$ goes to infinity.
On the other hand, by Lemma \ref{BoundAverage}, we have
$$\Bigl(\frac{1}{N^{2}}\sum_{i=0}^{N-1}\sum_{j=0}^{N-1}f_{1}(S^{i}x)f_{2}(T^{j}x)\widehat{f}_{3}(S^{i}T^{j}x)\Bigr)^{4}\leq \vert S_N(\wh{f}_3,x)\vert.$$

By Lemma \ref{ConvergenceContinuous}, the right hand side converges to $$\normm{\wh{f}_3}_{\mu,S,T}^4\leq \Bigl(\normm{\wh{f}_3-f_3}_{\mu,S,T}+\normm{f_3}_{\mu,S,T}\Bigr)^4 \leq \|f_3-\wh{f}_3\|_{1}^{4}\leq \epsilon$$
as $N$ goes to infinity. We deduce that a.e. $$\limsup_{N\to \infty} \left |\frac{1}{N^2}\sum_{i,j=0}^{N-1} f_1(S^ix)f_2(T^jx)f_3(S^iT^jx)\right |\leq 2\epsilon.$$
Since $\epsilon$ is arbitrary, we have that this average goes to 0 a.e.
\end{proof}

\appendix
\section{Facts about measure preserving systems}
In this appendix, we describe some concepts we use through the paper.

\subsection{Ergodic decomposition of a measure}

Let $(X,\mu,G_0)$ be measure preserving system and let $\I$ be the $\sigma$-algebra of invariant sets.  Let $x\to \mu_x$ be a regular version of {\it conditional measures} with respect to $\I$. This means that the map $x\mapsto \mu_x$ is $\I$-measurable and \[ \mathbb{E}(f|\I)(x)=\int f d\mu_x\quad  \mu \text{a.e. } x\in X \]

The ergodic decomposition of $\mu$ under $G_0$ is $\mu=\int_X \mu_x d\mu(x)$ and $\mu$ a.e. the system $(X,\mu_x,G_0)$ is ergodic.

\subsection{Conditional expectation and disintegration of a measure}
Let $\pi\colon Y$ $\to X$ be a factor map between the measure preserving systems $(Y,\mu,G_0)$ and $(X,\nu,G_0)$ and let $f\in L^2(\mu)$. The {\it conditional expectation} of $f$ with respect to $X$ is the function $\mathbb{E}(f|X)\in L^2(\nu)$ defined by the equation \[\int_X \mathbb{E}(f|X)\cdot g d\nu=\int_Y f \cdot g\circ\pi d\mu \quad \text{ for every } g \in L^2(\nu). \]
\begin{thm}
Let $\pi\colon Y\to X$ be a factor map between the measure preserving systems $(Y,\mu,G_0)$ and $(X,\nu,G_0)$. There exists a unique measurable map $X\to M(Y)$, $x\mapsto \mu_x$ such that
\begin{equation}
\mathbb{E}(f|X)(x)=\int f d\mu_x
\end{equation} for every $f\in L^1(\mu)$.
\end{thm}

 We say that $\mu=\int_X \mu_x d\nu(x)$ is the {\it disintegration} of $\mu$ over $\nu$.

\end{document}